\documentclass[11pt]{article}

\usepackage[T1]{fontenc}
\usepackage[utf8]{inputenc}
\usepackage[serif]{libertinus-type1}
\usepackage{microtype}
\usepackage{amsmath,amssymb,amsthm,mathtools}
\usepackage[a4paper,left=27mm,right=27mm,top=23mm,bottom=25mm,
            headheight=14pt,headsep=12pt,footskip=28pt]{geometry}
\usepackage{xcolor}
\usepackage{fancyhdr}
\usepackage{titlesec}
\usepackage{tocloft}
\usepackage{enumitem}
\usepackage[most]{tcolorbox}
\usepackage[colorlinks=true,
            linkcolor=teal!55!black,
            citecolor=teal!55!black,
            urlcolor=teal!55!black,
            pdfauthor={Self-contained research note},
            pdftitle={A Counting Lower Bound for Maximum Odd Induced Subgraphs},
            pdfsubject={Graph theory research note}]{hyperref}

\newcommand{\fo}{f_{\mathrm{o}}}
\newcommand{\foe}{f_{\mathrm{oe}}}
\newcommand{\Aell}{\mathcal{A}_{\ell}}
\newcommand{\cstar}{c_{*}}
\newcommand{\Acal}{\mathcal A}
\newcommand{\1}{\mathbf 1}
\newcommand{\one}{\mathbf{1}}
\newcommand{\E}{\mathbb E}

\newtheorem{theorem}{Theorem}[section]
\newtheorem{lemma}[theorem]{Lemma}
\newtheorem{corollary}[theorem]{Corollary}

\theoremstyle{remark}
\newtheorem{remark}[theorem]{Remark}

\titleformat{\section}
  {\Large\bfseries}
  {\thesection}
  {1.15em}
  {}
\titlespacing*{\section}{0pt}{2.5ex plus .8ex minus .2ex}{1.2ex}

\setlist[enumerate]{leftmargin=2.2em,itemsep=0.25em,topsep=0.35em}

\title{Weighted Counting Formula and $2n/21$ Lower Bound \\[0.25em]  for Induced Subgraphs with Prescribed Degree Parities}
\author{Gregory Gutin\thanks{Department for Computing, Security and Mathematics, Royal Holloway University of London, {\tt g.gutin@rhul.ac.uk}, and School of Mathematical Sciences and LPMC, Nankai University.}
\hspace{2mm} Yiming Hao\thanks{School of Mathematical Sciences and LPMC, Nankai University. {\tt  
1120230031@mail.nankai.edu.cn}.}
\hspace{2mm} Yacong Zhou\thanks{Shenzhen Institutes of Advanced Technology, Chinese Academy of Sciences. {\tt yacong.zhou96@gmail.com}.}
}

\begin{document}
\maketitle

\begin{abstract}
Let $G=(V,E)$ be a finite simple graph of order $n\geq 1$, and let $\ell:V\to\{0,1\}$ be a prescribed parity labeling. A set $S\subseteq V$ is called $\ell$-admissible if $d_S(v)\equiv \ell(v)\pmod 2$ for every $v\in S$, where $d_S(v)=|N_G(v)\cap S|$. Let $h_\ell(G)$ be the maximum order of an $\ell$-admissible set and let $f_{\rm oe}(G)=\min_\ell h_\ell(G)$. For $x\in\mathbb R$, define the weighted counting polynomial $$
  M_{\ell,x}(G)=\sum_{S\in {\cal A}_\ell(G)}x^{|S|}, $$ where ${\cal A}_\ell(G)$ is the collection of all $\ell$-admissible sets in $G$. For $R\subseteq V$, let $z_\ell(R)$ be the number of vertices $v\in V\setminus R$ for which $d_R(v)\equiv\ell(v)\pmod 2$. We prove the exact identity $$
  M_{\ell,x}(G)
  =2^{-n}\sum_{R\subseteq V}
  x^{|R|}(2+x)^{z_\ell(R)}(2-x)^{n-z_\ell(R)-|R|}. $$ If $G$ has no isolated vertices, then, for every $\ell$ and every $x\in(0,2)$, $
  M_{\ell,x}(G)>x^{n/2}(4-x^2)^{n/4}. $ Combining this estimate with a binary-entropy upper bound and optimizing $x$ gives $$
  f_{\rm oe}(G)>c_*n>\frac{2n}{21}, $$ where $c_*\approx0.095862615$. Ferber and Krivelevich (Adv. Math. 2022) proved that $h_{\mathbf{1}}(G)\ge 10^{-4}n$, where $\mathbf{1}$ is the all-one labeling. Since $h_{\mathbf{1}}(G)\ge f_{\rm oe}(G)$, our result improves coefficient in their bound by almost three orders of magnitude, and does so simultaneously for every labeling.
\end{abstract}

\vspace{-0.35em}

\section{Introduction}
\label{sec:introduction}

Throughout the paper, $G=(V,E)$ is a finite simple graph and ${n=|V|\geq 1}$. A graph is called \emph{odd} if every vertex has odd degree.  Let $\fo(G)$ denote the maximum order of an odd induced subgraph of $G$. A long-standing conjecture, described by Caro~\cite{Caro1994} as part of graph-theoretic folklore, 
asserted that there is an absolute constant $c>0$ such that every $n$-vertex graph without isolated vertices satisfies $\fo(G)\ge cn$.  Caro~\cite{Caro1994} obtained an $\Omega(\sqrt n)$ lower
bound, while Scott~\cite{Scott1992} established an $\Omega(n/\log n)$ bound. Ferber and Krivelevich~\cite{FK} resolved the conjecture by proving
\(
  \fo(G)\ge 10^{-4}n
\)
for every graph $G$ without isolated vertices.

It is natural to ask for the largest universal coefficient.  Define
\[
  c_{\mathrm{o}}
  :=\inf\left\{\frac{\fo(G)}{|V(G)|}:G\text{ is a graph without
  isolated vertices}\right\}.
\]
An example noted by Caro~\cite{Caro1994} gives $c_{\mathrm{o}}\le 2/7$.  For graphs of
maximum degree at most four, Ai, Guo, Gutin, Hao, and Yeo~\cite{Ai+2025}
proved the matching bound $\fo(G)\ge 2n/7$, and showed that it is sharp in
that class.  
{ 
Graphs considered in
\cite{BWW1997,Hou2018,Rao2022} 
do not lower the bound $c_{\mathrm{o}}\le 2/7.$
This raises a natural question: Is $c_{\mathrm{o}}= 2/7$ ?
}

The present paper works with the more general prescribed-parity setting first introduced 
by Gutin and Yeo \cite{GutinYeo2022}. 
For $S\subseteq V$ and $v\in V$, write
$d_S(v):=|N_G(v)\cap S|.$ Given a labeling
$\ell:V\to\{0,1\}$, a set $S\subseteq V$ is called \emph{$\ell$-admissible}
if
$
  d_S(v)\equiv \ell(v)\pmod 2 \quad\text{for every }v\in S.
$
The empty set is regarded as $\ell$-admissible.  Let
$$
  \Aell(G):=\{S\subseteq V:S\text{ is $\ell$-admissible}\},
  \qquad
  h_\ell(G):=\max\{|S|:S\in\Aell(G)\},
$$
and define the worst-labeling parameter
$
  \foe(G):=\min_{\ell:V\to\{0,1\}}h_\ell(G).
$
If $\one$ denotes the all-one labeling, then
$h_{\one}(G)=\fo(G)$, and hence
$
  \foe(G)\le \fo(G).
$
The parameter $\foe$ was studied by Ai, Guo, Gutin, Hao, and
Yeo~\cite{Ai+2026}, who established universal linear bounds and a
number of structural results for prescribed degree parities.

Our approach is based on weighted counting.  For any $x\in\mathbb R$ and labeling $\ell: V\to \{0,1\}$, let
\[
  M_{\ell,x}(G):=\sum_{S\in\Aell(G)}x^{|S|}.
\]
{Note that $M_{\ell,x}(G)$ can be viewed as a weighted counting where every set $S\in\Acal_{\ell}(G)$ is assigned weight $x$ raised to $|S|$
}. Thus $M_{\ell,1}(G)=|\Aell(G)|$, while $M_{\one,1}(G)$ is the number of
vertex sets inducing odd subgraphs, with the empty set counted once.  For
$R\subseteq V$, define
\[
  z_\ell(R)
  :=\bigl|\{v\in V\setminus R:d_R(v)\equiv\ell(v)\pmod 2\}\bigr|.
\]
In Section \ref{sec:1},  we prove the exact identity
\begin{equation}\label{eq:weighted-identity-intro}
  M_{\ell,x}(G)
  =2^{-n}\sum_{R\subseteq V}
    x^{|R|}(2+x)^{z_\ell(R)}(2-x)^{n-z_\ell(R)-|R|}.
\end{equation}
When $G$ is nonempty and has no isolated vertices, the same section uses
\eqref{eq:weighted-identity-intro} and strict Jensen convexity to show that
\begin{equation}\label{eq:jensen-lower-intro}
  M_{\ell,x}(G)>x^{n/2}(4-x^2)^{n/4}
  \qquad (0<x<2).
\end{equation}

{In Section \ref{sec:2}, by applying \eqref{eq:jensen-lower-intro} and a standard entropy upper bound with an optimized $x$, we give a short proof for the following main result, where $\cstar\in (0.0958626150,0.0958626151)$ and will be defined rigorously in Section \ref{sec:2}. }

\begin{theorem}
\label{thm:linear}
If $G$ is a graph without isolated vertices, then
$
  \foe(G)>c_*n>\frac{2n}{21}.
$
\end{theorem}

{Thus, $\fo(G)\ge  \foe(G)>c_*n$, which improves coefficients in the bounds for $\fo(G)$ in~\cite{FK}  and for $\foe(G)$ in ~\cite{Ai+2026} by almost three orders of magnitude. 

Finally, in Section~\ref{sec:disc}, we discuss some open problems on the topic.}


\section{A Weighted Counting Formula and a Lower Bound for $M_{\ell,x}(G)$}\label{sec:1}

\begin{theorem}\label{thm:exact}
For every graph $G=(V,E)$ on $n$ vertices, function $\ell: V\to\{0,1\}$ and { $x\in \mathbb{R}$}, we have
\[
 M_{\ell,x}(G)=2^{-n}\sum_{R\subseteq V}x^{|R|}(2+x)^{z_{\ell}(R)}(2-x)^{n-z_{\ell}(R)-|R|}, 
\]
{with the convention $0^0=1$}.
\end{theorem}
\begin{proof}
For a fixed set $S\subseteq V$, the indicator that $S$ is an $\ell$-admissible set is
\[
  \1_{\{S\in\Acal_{\ell}(G)\}}
  =\prod_{v\in S}\frac{1+(-1)^{d_S(v)+\ell(v)}}{2}=2^{-|S|}\prod_{v\in S}\left(1+(-1)^{d_S(v)+\ell(v)}\right).
\]
Indeed, the factor indexed by $v$ equals $1$ when $d_S(v)$ has the same parity as $\ell(v)$, $0$ otherwise. Summing over all $S$ and expanding the product, we obtain

\begin{eqnarray*}
     M_{\ell,x}(G)&=&\sum_{S\in\Acal_{\ell}(G)}x^{|S|}=\sum_{S\subseteq V} x^{|S|} \cdot \1_{\{S\in\Acal_{\ell}(G)\}}\\
  &=&\sum_{S\subseteq V} \left(\frac{x}{2}\right)^{|S|}\prod_{v\in S}\left(1+(-1)^{d_S(v)+\ell(v)}\right)\\
  &=&\sum_{S\subseteq V}\left(\frac{x}{2}\right)^{|S|}
    \sum_{T\subseteq S}(-1)^{\sum_{v\in T}(d_S(v)+\ell(v))}
\end{eqnarray*}
As $\sum_{v\in T}d_S(v)=2e(G[T])+\sum_{v\in T}d_{S\setminus T}(v)$, the above equality gives
\begin{equation}\label{eq:expansion}
     M_{\ell,x}(G)
  =\sum_{S\subseteq V}\left(\frac{x}{2}\right)^{|S|}
    \sum_{T\subseteq S}(-1)^{\sum_{v\in T}(d_{S\setminus T}(v)+\ell(v))}
\end{equation}
Let $R=S\setminus T$. Thus, $S=R\mathbin{\dot\cup}T$.
Re-indexing the sum in \eqref{eq:expansion} by disjoint pairs $(R,T)$ yields
\begin{eqnarray*}
  M_{\ell,x}(G)
  &=&\sum_{R\subseteq V}\left(\frac{x}{2}\right)^{|R|}
    \sum_{T\subseteq V\setminus R}\left(\frac{x}{2}\right)^{|T|}(-1)^{\sum_{v\in T}(d_{R}(v)+\ell(v))}\\
    &=&\sum_{R\subseteq V}\left(\frac{x}{2}\right)^{|R|}
{ \prod_{v\in V\setminus R}
   \left(1+\frac{x}{2}(-1)^{d_R(v)+\ell(v)}\right), }
\end{eqnarray*}
where if $d_R(v)\equiv \ell(v)\pmod 2$, the corresponding factor is $1+\frac{x}{2}$, and otherwise $1-\frac{x}{2}$. Thus,
\begin{eqnarray*}
 M_{\ell,x}(G)&=&\sum_{R\subseteq V}\left(\frac{x}{2}\right)^{|R|}\left(1+\frac{x}{2}\right)^{z_{\ell}(R)}\left(1-\frac{x}{2}\right)^{n-z_{\ell}(R)-|R|}\\
 &=& 2^{-n}\sum_{R\subseteq V}x^{|R|}(2+x)^{z_{\ell}(R)}(2-x)^{n-z_{\ell}(R)-|R|}
\end{eqnarray*}
completing the proof. 
\end{proof}

We now establish a lower bound for $M_{\ell,x}(G)$ when $G$ has no isolated vertices.

\begin{theorem}\label{thm:mean}
{Let $n\geq 1$}. For any $G=(V,E)$ on $n$ vertices with no isolated vertices, $\ell: V\to \{0,1\}$ and $x\in (0,2)$, let $R$ chosen uniformly from the $2^n$ subsets of $~V$. Then, $\E[|R|]=n/2$, $\E [z_{\ell}(R)]= n/4$ and thus 
\[M_{\ell,x}(G)>x^{n/2}(4-x^2)^{n/4}.\] 
\end{theorem}
\begin{proof}
{Note that as $G$ has no isolated vertices, $n\geq 2$.} $\E[|R|]=n/2$ clearly holds as $R$ is chosen uniformly at random from $V$. We now show $\E [z_{\ell}(R)]= n/4$. For any fixed vertex $v\in V$, there are exactly $2^{n-1}$ subsets of $V$ that do not contain $v$. As $v$ is not an isolated vertex one can observe that among these $2^{n-1}$ subsets, exactly one half of them (and therefore $2^{n-2}$ subsets) contain an odd (or even) number of neighbours of $v$. Thus, 
\[\E[z_\ell(R)]=\sum_{v\in V}\mathbb{P}\bigl(v\notin R\text{ and $d_R(v)\equiv \ell(v)\pmod 2$}\bigr)=\sum_{v\in V} \frac{2^{n-2}}{2^n}=n/4.\] 

Now, by Theorem \ref{thm:exact}, 
\begin{eqnarray*}
    M_{\ell,x}(G)&=&\E\left[x^{|R|}(2+x)^{z_{\ell}(R)}(2-x)^{n-z_{\ell}(R)-|R|}\right]\\
    &=&\E\left[e^{|R|\ln(x)+z_\ell(R)\ln(2+x)+(n-z_{\ell}(R)-|R|)\ln(2-x)}\right],
\end{eqnarray*}
where we use the fact that $x\in (0,2)$ (and therefore $x$, $2+x$ and $2-x$ are all positive numbers). { Note that as $n\geq 2$, $X(R)=|R|\ln(x)+z_\ell(R)\ln(2+x)+(n-z_{\ell}(R)-|R|)\ln(2-x)$ is not a constant. In fact, if $x\neq  1$, then $X(V)=n\ln(x)$, and for any vertex $v\in V$, $X(V\setminus \{v\})=(n-1)\ln(x)+\ln(2-x)$ or $(n-1)\ln(x)+\ln(2+x)$. Thus, $X(V\setminus\{v\})-X(V)$ is either $\ln\left(\frac{2-x}{x}\right)$ or $\ln\left(\frac{2+x}{x}\right)$, and neither is zero when $x\neq 1$. If $x= 1$, we have $X(R) = z_\ell(R) \ln 3$ for all $R\subseteq V$. Since $z_\ell(V) = 0$, it is enough to find an R with $z_\ell(R) > 0$: one can use $R= \emptyset$ if some vertex has label 0; if $\ell= \1$, we can choose any vertex $u\in V$ and take $R= \{u\}$ and therefore $z_\ell(R)\geq |N_G(u)|>0$.

Since $|R|\ln(x)+z_\ell(R)\ln(2+x)+(n-z_{\ell}(R)-|R|)\ln(2-x)$ is not a constant and  \(e^x\) is strictly convex, Jensen's inequality gives}
\begin{eqnarray*}
     M_{\ell,x}(G)&=&\E\left[e^{|R|\ln(x)+z_\ell(R)\ln(2+x)+(n-z_{\ell}(R)-|R|)\ln(2-x)}\right]\\
   &>& e^{\E [|R|\ln(x)+z_\ell(R)\ln(2+x)+(n-z_{\ell}(R)-|R|)\ln(2-x)]}\\
   &=& e^{\E[|R|]\ln(x)+\E[z_\ell(R)]\ln(2+x)+\E[n-z_{\ell}(R)-|R|]\ln(2-x)}\\
   &=& e^{\frac{n}{2}\ln(x)+\frac{n}{4}\ln(2+x)+\frac{n}{4}\ln(2-x)}\\
   &=& x^{n/2}(4-x^2)^{n/4},
\end{eqnarray*}
completing the proof.
\end{proof}

{Let $\ell:V\to\{0,1\}$ be a labeling. As $M_{\ell,1}(G)$ equals the number of $\ell$-admissible sets (recall that empty set is counted once), by setting $x=1$ we have the following lower bound for the number of $\ell$-admissible sets in $G$.

\begin{corollary}
If $G=(V,E)$ is an $n$-vertex graph without isolated vertices, then  $M_{\ell,1}(G)>3^{n/4}   \mbox{ for every } \ell:V\to\{0,1\}.$
\end{corollary}
}

\section{The Definition of $\cstar$ and Proof of Theorem~\ref{thm:linear}}\label{sec:2}
To state the resulting coefficient $\cstar$, let
\[
H_2(t):=-t\log_2 t-(1-t)\log_2(1-t)
\qquad (0<t<1)
\]
be the binary entropy function, and define
\[
\Phi(t)
:=\frac{1-2t}{4}\log_2(1-2t)
-\frac{1-t}{2}\log_2\!\left(\frac{1-t}{2}\right)
\qquad (0<t<1/2).
\]
Then, the following holds.

\begin{lemma}
	\label{lem:unique-root}
	The equation $H_2(t)=\Phi(t)$ has exactly one solution in $(0,1/2)$.
\end{lemma}

\begin{proof}
	Let $h(t)=H_2(t)-\Phi(t)$.  {Since $t\log t\to 0$ as $t\to 0^{+}$,} the function $h$ extends continuously to
	$[0,1/2]$, with
	\[
	h(0)=-\frac12
	\qquad\text{and}\qquad
	h\!\left(\frac12\right)=\frac12.
	\]
	Moreover,
	\[
	h'(t)
	=\log_2\!\left(\frac{1-t}{t}\right)
	-\frac12\log_2\!\left(\frac{1-t}{2(1-2t)}\right)
	=\frac12\log_2\!\left(\frac{2(1-t)(1-2t)}{t^2}\right).
	\]
	Hence $h'(t)=0$ if and only if
	$3t^2-6t+2=0$.  The unique critical point in $(0,1/2)$ is
	$
	\tau=1-\frac{1}{\sqrt3}.
	$
	{Moreover $h'>0$ on $(0,\tau)$ and 
		$h'<0$ on $(\tau,\tfrac12)$. Hence $h$ increases from $h(0)<0$ to its maximum $h(\tau)\ge h(\tfrac12)>0$, crossing zero exactly once, and then decreases to $h(\tfrac12)>0$, remaining positive. This completes the proof.}
\end{proof}

Thus, we can define $c_*$ as the unique solution from Lemma~\ref{lem:unique-root}.
A numerical evaluation {via bisection method} gives 
{\begin{equation}\label{cstarest}
		\cstar\in (0.0958626150,0.0958626151).
\end{equation}} Now, we are ready to prove our main result:

\begin{proof}[Proof of Theorem \ref{thm:linear}]
Set $x=\sqrt\frac{2(1-2\cstar)}{1-\cstar}$ (see Remark \ref{remark1} for a reason to choose this value of $x$). By \eqref{cstarest}, $x\in (1,2)$. Suppose, for a contradiction, that $\foe(G)\le \cstar n$. 
Then there exists a labeling $\ell:V\to \{0,1\}$ such that every $\ell$-admissible set has order at most \(\lfloor \cstar n\rfloor\), and thus as $x>1$, we have
$$
   M_{\ell,x}(G)
   \le \sum_{k=0}^{\lfloor \cstar n\rfloor}\binom nk x^k\le x^{\cstar n}\sum_{k=0}^{\lfloor \cstar n\rfloor}\binom nk \le x^{\cstar n}{ (\cstar) }^{-\cstar n}(1-\cstar)^{-(1-\cstar) n} 
   = 2^{H_2(\cstar) n}x^{\cstar n}
$$
where the third inequality holds as $\sum_{k=0}^{\lfloor \cstar n\rfloor}\binom nk\cstar^{{\cstar n}}(1-\cstar)^{(1-\cstar)n} \le \sum_{k=0}^{\lfloor \cstar n\rfloor}\binom nk\cstar^k(1-\cstar)^{n-k}\leq 1$ (which holds as $\cstar<1/2$). Thus, as $H_2(\cstar)=\Phi(\cstar)$ we have
\begin{eqnarray*}
     M_{\ell,x}(G)&\leq& (2^{\Phi(\cstar)}x^{\cstar})^n\\
     &=& \left((1-2\cstar)^{\frac{1-2\cstar}{4}}x^{\cstar }\big/\left(\frac{1-\cstar}{2}\right)^{\frac{1-\cstar}{2}}\right)^n\\
    &=& \left(\left(\frac{2(1-2\cstar)}{1-\cstar}\right)^{\frac{1-2\cstar}{4} }\left(\frac{2}{1-\cstar}\right)^{\frac{1}{4}}x^{\cstar}\right)^n\\
    &=& \left(x^{\frac{1-2\cstar}{2} }\left(4-x^2\right)^{\frac{1}{4}}x^{\cstar}\right)^n\\
    &=& x^{\frac{n}{2} }\left(4-x^2\right)^{\frac{n}{4}},
\end{eqnarray*}
where the third equality uses the fact that $4-x^2=\frac{2}{1-\cstar}$. However, this contradicts Theorem \ref{thm:mean} and therefore $\foe(G)> \cstar n$. In addition, as $2/21{<0.0953}<\cstar$, $\foe(G)> 2n/21$. \end{proof}

\begin{remark}\label{remark1}
	The choices of $\cstar$ and $x$ may look a bit strange. In fact they are forced, and are determined in the order opposite to the one in which they are used. Fix $c\in(0,1/3)$ and suppose, for a contradiction, that $\foe(G)\le cn$ (we don't lose anything by doing that, as $c_{\rm o}\leq 2/7<1/3$ anyway), so that
	some labeling $\ell$ admits no $\ell$-admissible set of order exceeding
	$\lfloor cn\rfloor$. For every $x\in[1,2)$, the entropy estimate and
	Theorem~\ref{thm:mean} give
	\[
	\sqrt[n]{M_{\ell,x}(G)}\le 2^{H_2(c)}x^{c},
	\qquad
	\sqrt[n]{M_{\ell,x}(G)}> x^{1/2}\bigl(4-x^{2}\bigr)^{1/4},
	\]
	so, taking logarithms, we obtain a contradiction---and hence $\foe(G)>cn$---as
	soon as
	\[
	H_2(c)\;\le\;\Bigl(\tfrac12-c\Bigr)\log_2 x+\tfrac14\log_2\bigl(4-x^{2}\bigr)
	\;=:\;F_c(x).
	\]
	As it suffices that this holds for \emph{some} $x$, the largest value
	of $c$ reachable by the method is the largest one with
	$H_2(c)\le\max_x F_c(x)$. Since $c<1/2$, we have $F_c(x)\to-\infty$ as
	$x\to 0^{+}$ and as $x\to 2^{-}$, while $F_c'(x)=0$ if and only if
	$(1-2c)(4-x^{2})=x^{2}$, that is, if and only if
	\[
	x^{2}=\frac{2(1-2c)}{1-c}.
	\]
	This unique critical point is therefore the maximizer, and it lies in $(1,2)$
	precisely as $c<1/3$. This is the reason for the choice
	$x=\sqrt{2(1-2\cstar)/(1-\cstar)}$. Substituting it into $F_c$ yields
	\[
	\max_{x}F_c(x)=\frac{1-2c}{4}\log_2(1-2c)-\frac{1-c}{2}\log_2\frac{1-c}{2}
	=\Phi(c).
	\]
	Thus the method proves $\foe(G)>cn$ for exactly those $c\in(0,1/3)$ with
	$H_2(c)\le\Phi(c)$. By Lemma~\ref{lem:unique-root} and its proof, $H_2-\Phi$ is
	negative on $(0,\cstar)$ and positive on $(\cstar,1/2)$, so this set is
	$(0,\cstar]$ and the optimum is attained at the equality $H_2(c)=\Phi(c)$, which is how $\cstar$ is defined. 
\end{remark}

\section{Open Problems}\label{sec:disc}

Similarly to $c_{\mathrm{o}}$, let us introduce the universal
worst-labeling constant
$$
  c_{\mathrm{oe}}
  :=\inf\left\{\frac{\foe(G)}{|V(G)|}:G\text{ is a graph without
  isolated vertices}\right\}.
$$
As we discussed in Section \ref{sec:introduction}, $c_{\rm o}\le 2/7$. 
Hence, $\cstar \le c_{\rm oe}\le c_{\rm o}\le 2/7$. Note that the inequality  $\cstar \le c_{\rm oe}$ is non-strict: although
Theorem~\ref{thm:linear} is strict for every individual graph, taking an
infimum need not preserve strictness.  Determining either
$c_{\mathrm{o}}$ or $c_{\mathrm{oe}}$ remains open.  In particular, it is
natural to ask whether $c_{\mathrm{o}}=2/7$ and how much smaller the
worst-labeling constant can be.  

Note that $c_*$ is the largest coefficient obtainable from
the entropy upper bound and the Jensen lower bound in their present
forms.  Any improvement beyond $c_*$ must therefore use additional
information, for example a sharper upper bound on the size distribution of
$\ell$-admissible sets, or a lower bound that exploits more than the first
moments of $|R|$ and $z_\ell(R)$.

Ai et al. \cite{Ai+2026} conjectured
that
$
  \foe(G)\ge\frac12\fo(G)
$
for every graph $G$.  A proof would provide a direct structural relation
between the two universal constants, but it would not by itself determine
either one.


\subsection*{Acknowledgement} YH was partially supported by China Scholarship Council (No.202406200160).

\end{document}